\documentclass {article}

\usepackage{amsmath}
\usepackage{amssymb}
\usepackage{amsfonts}
\usepackage{amsthm}
\usepackage{color}
\newtheorem{theorem}{Theorem}
\newtheorem{corollary}[theorem]{Corollary}

\newtheorem{lemma}[theorem]{Lemma}
\newtheorem{proposition}[theorem]{Proposition}
\newtheorem{remark}[theorem]{Remark}

\date{\today}
\begin{document}

\title{Invariant measures for   stochastic damped 2D Euler equations}

\author{Hakima Bessaih
\footnote{University of Wyoming, Department of Mathematics, Dept. 3036, 1000
East University Avenue, Laramie WY 82071, United States, bessaih@uwyo.edu}
\and Benedetta Ferrario
\footnote{(Corresponding author) Universit\`a di Pavia, Dipartimento di Matematica, via
  Ferrata 5, 27100 Pavia, Italy,
benedetta.ferrario@unipv.it, phone (+39)0382 985655, fax (+39)0382 985602}
}

\maketitle

\begin{abstract}
We study the two-dimensional Euler equations, damped by a linear term and 
 driven by an additive noise. 
The existence of  weak solutions has already been studied;
pathwise uniqueness 
is known for solutions that have vorticity in $L^\infty$.
In this paper, we prove the Markov property and then  the existence of an
invariant measure in the space $L^\infty$ by means of 
 a 
Krylov-Bogoliubov's type method, working with  the weak$\star$  
and the bounded  weak$\star$   topologies in $L^\infty$.
 \end{abstract}
\noindent
{\bf MSC2010}: 60H15, 37L40, 47D07, 76B03, 60J99.\\
{\bf Keywords}: Stochastic Euler equations, vorticity formulation,
Markov processes, invariant measures,  dissipative dynamical systems.

\bigskip

\section{Introduction}
Two dimensional hydrodynamics  is  largely studied
from the theoretical as well as from the applied point of view.
Both the analysis of individual solutions or statistical solutions have been developped.
In particular  turbulence theory, which analyzes the equations of motion of a fluid
by introducing statistical means, asks for
existence/uniqueness of statistically stationary solutions.
They describe the motion of fluids at equilibrium, for large time.

For the bidimensional  equations of viscous fluids,
that is  the Navier-Stokes equations, forced by a random forcing term
existence and uniqueness of invariant measures have been proved
under many different assumptions on the noise term
(see, among the others, \cite{HM} and references therein).

Moreover,
these equations with a weaker dissipation 
have been considered more recently by Constantin,  Glatt-Holtz and Vicol \cite{CGHV}
proving existence and uniqueness of invariant measures;
they are called the fractionally dissipated Euler equations.

On the other hand, for the   stochastically forced
2D Euler equations with a linear damping, which is a wave-number independent dissipation,
the only so far known results on  the longtime behavior are 
through their weak random attractors and stationary solutions
(see \cite{Bessaih2000, BF,Bessaih2008,BesFer13}).
These equations are given by 
\begin{equation}\label{Eu}
\begin{cases}
du+(u\cdot\nabla)u \ dt +\gamma u \ dt +\nabla p\ dt = dW\\
\nabla \cdot u =0
\end{cases}
\end{equation}
The unknowns are the velocity vector $u=u(t,x)$ and the pressure
$p=p(t,x)$;
here $t$ is the time variable and $x\in D\subset \mathbb R^2$ the space variable.
$W=W(t,x)$ is a given Wiener process. We assume $\gamma> 0$.
With respect to the vorticity $ \xi=\nabla^\perp \cdot u\equiv
  \partial_1 u_2- \partial_2 u_1$ they are
\begin{equation}\label{vor}
\begin{cases}
 d\xi+ u\cdot\nabla\xi \ dt +\gamma\xi\ dt = dW^{curl}\\
 \xi=\nabla^\perp \cdot u,  \quad \nabla \cdot u =0
 \end{cases}
 \end{equation}
with $W^{curl}=\nabla^\perp\cdot W$.

When
$\gamma=0$ the above are the 
Euler equations governing the motion of an incompressible inviscid fluid 
that have been extensively studied.
When $\gamma~>~0$, the linear damping, although not regularizing, introduces 
some dissipative feature, discussed in
\cite{boffetta, gal}.   Kupianinen in  \cite{Ku} points out how these
randomly forced damped Euler equations are related to 2D turbulence
theory and to the viscous case (see also \cite{BesFer13}); moreover,
interesting scaling limits on the vanishing  viscosity and/or the damping are discussed,
giving some open conjectures on the limits problems.

We recall that in \cite{Bessaih2008} 
the existence of stationary solutions 
to these stochastic damped Euler equations has been proved
in the space $L^2(D)$ for the vorticity.  In particular  this is a space
where the uniqueness does not hold. Let us recall  that there is no need
to define the associated  transition semigroup in order to define stationary solutions.
Hence, having stationary solutions is a weaker result than having
an invariant measure where a proper dynamics is needed.
Here we improve that result by defining a transition semigroup in the space  $L^\infty(D)$,
which is the space where uniqueness is proved for equation \eqref{vor}.
The drawback of working in the space  $L^\infty(D)$ is that it
is not separable, and weak$\star$ measurability and strong measurability do not coincide.
In this paper we prove that the transition semigroup
is sequentially weakly$\star$ Feller and Markov in  $L^\infty(D)$ equipped with the
bounded  weak$\star$ topology.
Then, we construct an invariant measure  by means of  Krylov-Bogoliubov's technique but
dealing with weak$\star$ topologies,
in a similar way as done by Maslowski and Seidler in \cite{MS}
(however they worked in a separable Hilbert space). 

As far as we know, this is the first result for the damped Euler equation \eqref{vor}
and the first result for any fluid dynamic equation in a non separable
space setting, like  $L^\infty(D)$.
We hope that our method could be used to tackle other models with similar problems. 

The paper  is organized as follows. In Section 2, we introduce 
the functional spaces and assumptions.
The space $L^\infty(D)$ with its various topologies is  described in some detail
in subsection 2.2.
A particular attention will be devoted to the bounded weak$\star$ topology;
this is a crucial point that will be used in the Krylov-Boguliobov's technique
for the passage to the limit in order to get the invariant measure.
We also recall some well posedness results, that are not new but contain
some improvements for the measurability of the solutions in $L^\infty(D)$.
In Section 3, we prove the 
continuous dependence of the vorticity solution with respect to the initial 
data and a spatial regularity result in the Sobolev space 
$[W^{1,4}(D)]^{2}$.  This leads to
 the   ''weak'' Feller property for  the transition Markov 
semigroup that is defined afterwards. 
In  Section 4, we prove the Markov property in the space $L^\infty(D)$
for system \eqref{vor}. In particular, 
we first prove the Markov property in $W^{1,4}(D)$ and then conclude by 
a density argument. Finally in Section 5 we prove existence of an invariant measure; 
this is the only part in which the assumption $\gamma>0$ is required,
otherwise all the previous results hold for any $\gamma\ge 0$.

\section{Preliminaries and assumptions}

\subsection{Mathematical setting}

Let $D$ be the torus $\mathbb R^2\setminus \mathbb Z^2$.
This means that the spatial domain is a square  and periodic boundary conditions are assumed.
The results remain true in a bounded domain,  see \cite{Bardos, BF, Bes15}.

We define the space $H$ of periodic vector fields
which are square integrable, divergence free and have zero mean value on $D$.
This is a separable Hilbert space, with the $[L^2(D)]^2$-scalar
 product. 
We denote by $|\cdot|$ the $H$-norm and by $(\cdot,\cdot)$ the
$H$-scalar product;
 $|\cdot|_p$ is  the $[L^p(D)]^2$-norm.
\\
We define $V=[H^1(D)]^2\cap H$ and  
denote by $\|\cdot\|$ its norm.
\\
For $k \ge 1$ and  $p > 2$ we define
$V^{k,p}=[W^{k,p}(D)]^2\cap V$, being $W^{k,p}(D)$ the Sobolev space. We denote
by $\|\cdot\|_{k,p}$ the $V^{k,p}$-norm.
$V^{k_1,p}$ is a dense subspace of $V^{k_2,p}$ for $k_1>k_2$ and the
embedding is compact. For simplicity we write $V^k$ for $V^{k,2}$.

Let $V^\prime$ be the dual space of $V$ with respect to the $H$ scalar product.
Identifying $H$ with
its dual space $H'$, and $H'$ with the corresponding natural
subspace of the dual space $V'$, we have the  Gelfand triple
$V\subset H\subset V'$ with continuous dense injections. 
We denote the dual pairing between $u\in V$ and $v\in V^\prime$ by $\langle  u,v\rangle$.
When $v\in H$, we have $(u,v)=\langle u,v\rangle$. For other duality
pairings the spaces will be specified when necessary.

Let $b(\cdot,\cdot,\cdot): V\times V\times V\longrightarrow
\mathbb{R}$ be the continuous trilinear form defined as
$$
 b(u,v,z)=\int_{D}([u(x)\cdot\nabla] v(x))\cdot z(x)\, dx .
$$
It is well known that there exists a continuous bilinear operator
$B(\cdot,\cdot): V\times V\longrightarrow V'$ such that
$\langle  B(u,v),z\rangle =b(u,v,z),\ {\rm for}\ {\rm all}\ z\in V.$
By the incompressibility condition, for $u,v, z\in V$
we have (see, e.g., \cite{Temam})
\begin{equation} \label{incompress}
\langle B(u,v),z\rangle=- \langle B(u,z),v\rangle \quad   \mbox{\rm and}\quad    \langle B(u,v),v\rangle =0.
\end{equation}

Working on the torus we can develop the velocity and the vorticity in Fourier series, so to 
easily express the relationship between $u$ and $\xi$ (see, e.g.,
details in \cite{BesFer13}),
proving that for any $p\in [2,\infty)$
the norms $|\nabla u|_p$ and $|\xi|_p$ are equivalent  and
that the norm $| u|_{p}$ is bounded by the norm $|\xi|_p$.

As far as the stochastic part is concerned, we are given a complete probability space
$(\Omega, \mathcal F, \mathbb P)$ and  a sequence   
$\{\tilde \beta_j(t); t\ge 0\}_{j\in\mathbb N}$ of independent 
standard 1-dimensional Wiener processes defined on it. Then we consider a new sequence of i.i.d. Wiener processes defined for any time $t \in \mathbb R$:
\[
\beta_i(t)=\begin{cases}
        \tilde\beta_{2i-1}(t) & \text{ for } t\ge 0\\
        \tilde\beta_{2i}(-t)  & \text{ for } t\le 0\\
\end{cases}
\]
The noise forcing term in equation \eqref{Eu} is taken of the form 
\begin{equation}\label{serie-noise}
 W(t,x)=\sum_{i \in \mathbb N} c_i \beta_i(t) e_i(x) 
\end{equation}
for some  $c_i \in \mathbb R$ (see, e.g., \cite{DPZ}), where $\{e_i\}_i$ is a complete
orthonormal system of $H$.
We define the filtration $\{\mathcal F_t\}_{t\in \mathbb R}$ by 
$\mathcal F_t=\sigma\{ W(t_2)-W(t_1), -\infty<t_1<t_2\le t  \}$.

In the sequel we shall require $W$ to take values in the space
$C(\mathbb R;V^{k,\infty})$  for $k=2$ or $k=3$;
by Sobolev embedding we know that it is sufficient that for some
  $h>k+1$ the paths 
$W\in C(\mathbb R;V^{h})$ a.s.; a sufficient condition  for this is that
\begin{equation}\label{noise}
 \sum_i c_i^2 \|e_i\|^2_{V^{h}}<\infty.
\end{equation}

\subsection{The space $L^\infty(D)$}

To shorten notation we write $L^p$ for the space $L^p(D)$.  
The space $L^\infty$ is the dual of the space $L^1$; moreover the space $L^\infty$ is not 
separable whereas the space $L^1$ is separable. This is a crucial property which makes the analysis of the dynamics
\eqref{Eu} a delicate matter with respect to some issues. Indeed, main
results available in the literature about stochastic PDE's 
are based on the assumption that the state space is separable
(see e.g. \cite{DPZ,DPZ2}). 

We recall the meaning of convergence in $L^\infty$ with respect to 
the  weak$\star$ topology: 
$\xi_n \overset{\star}{\rightharpoonup}  \xi $ in $L^\infty$
means
\[
_{L^\infty}\langle \xi_n, \phi \rangle _{L^1} \to \
 _{L^\infty}\langle \xi, \phi \rangle_{L^1} \qquad 
 \forall \phi \in L^1 .
\]

Here we collect basic results on topologies and related Borelian
subsets of $L^\infty$ (see, e.g., \cite{Meg}).

We denote by $\mathcal T_n$, $\mathcal T_{bw\star}$, $\mathcal T_{w\star}$
the strong (or norm) topology, the bounded weak$\star$ topology and
the weak$\star$ topology of $L^\infty$, respectively.
We have that
\begin{equation}\label{top-incl}
\mathcal T_{w\star} \subsetneq \mathcal T_{bw\star}\subsetneq
\mathcal T_n .
\end{equation}
We recall that the bounded weak$\star$ topology is the finest topology
on $L^\infty$
that coincides with the weak$\star$ topology  on every norm bounded
subset of $L^\infty$. 

Let us note that
$f:L^\infty \to \mathbb R$  is $\mathcal T_{bw\star}$-continuous 
if and only if it is sequentially 
$\mathcal T_{w\star}$-continuous\footnote{The space 
$SC(L^\infty,\mathcal T_{w\star})$ of sequentially
  weakly$\star$ continuous functions is the space of all  functions
$f:L^\infty \to \mathbb R$ such that $f(\xi_n)\to f(\xi)$ if $\xi_n
  \rightharpoonup \xi$  weakly$\star$ in $L^\infty$, i.e. 
$\langle \xi_n,g\rangle \to \langle \xi,g\rangle $ for any $g \in L^1$.}. 
Indeed, set $K_n = \{\xi \in L^\infty : \|\xi\|_{L^\infty} \le n\}, n \in \mathbb N$, and note that 
$K_n$ are metrizable $ \mathcal T_{bw\star}$-compact spaces. 
If $f$ is $\mathcal T_{bw\star}$-continuous and $\xi_j \to \xi$ weakly$\star$, 
then for some $n$ we have $\xi_j,\xi \in K_n$; 
the weak$\star$ continuity of $f|_{K_n}$ implies $f(\xi_j) \to f(\xi)$. 
In the opposite direction, let $f$ be sequentially weakly$\star$
continuous. Then
$f|_{K_n}$ is weakly$\star$-continuous on any $K_n$ by metrizability
of the weak$\star$  topology on bounded subsets. If $U \subset \mathbb R$ 
is an arbitrary open set, then $f^{-1}(U)\cap K_n = (f|_{K_n})^{-1}(U)$ 
is $\mathcal T_{w\star}$-open in $K_n$, so $f^{-1}(U)$ is $\mathcal T_{bw\star}$-open 
and $\mathcal T_{bw\star}$-continuity of $f$ follows.

Denoting by $C(L^\infty,\mathcal T)$ the space of all functions
$f:L^\infty\to \mathbb R$ which are $\mathcal T$-continuous, we thus  have that
\[ 
C(L^\infty,\mathcal T_{w\star})\subsetneq 
C(L^\infty,\mathcal T_{bw\star})=SC(L^\infty,\mathcal T_{w\star})\subsetneq
C(L^\infty,\mathcal T_n) .
\]

We recall that by Alaoglu-Banach theorem, the set
$\{\xi\in L^\infty: \|\xi\|_{L^\infty}\le R\}$ is $\mathcal T_{w\star}$-compact. 
Hence it is also $\mathcal T_{bw\star}$-compact,
since the $\mathcal T_{w\star}$-compact subsets coincide with the 
 $\mathcal T_{bw\star}$-compact subsets.

As far as measurability with respect to these topologies is concerned, 
let us denote by $\mathcal B(\mathcal T)$ the $\sigma$-algebra of Borelian
subsets of $L^\infty$ w.r.t.  the a given topology $\mathcal T$. According to 
\eqref{top-incl}
 we have that
 $ \mathcal B(\mathcal T_{w\star}) \subseteq \mathcal B(\mathcal T_{bw\star})
 \subseteq  \mathcal B(\mathcal T_n)$.
 Moreover
\begin{lemma}
For the space $L^\infty$ we have
\[
\mathcal B(\mathcal T_{w\star})=\mathcal B(\mathcal T_{bw\star}).
\]
\end{lemma}
\begin{proof}
From \eqref{top-incl} it follows that
$\mathcal B(\mathcal T_{w\star})\subseteq \mathcal B(\mathcal T_{bw\star})$.
Let us show the reverse inclusion.

Recall that a  basis for the weak$\star$ topology $\mathcal T_{w\star}$ of $L^\infty$ is given by the
collection of all subsets
\begin{equation}\label{basis-w}
B(\eta;g_1,\ldots,g_m)=\{\xi\in L^\infty:|\langle \xi-\eta,g_i\rangle|<1 \text{
  for } i=1,\ldots,m\}
\end{equation}
for any $\eta \in L^\infty$, for any $m \in \mathbb N$ and $g_i \in L^1$
(see page 224 in \cite{Meg}), and
a basis for the bounded weak$\star$ topology $\mathcal T_{bw\star}$ of $L^\infty$ is given by the
collection of all subsets
\begin{equation}\label{basis-bw}
B(\eta;\{g_i\}_{i\in \mathbb N})=\{\xi\in L^\infty:|\langle \xi-\eta,g_i\rangle|<1 \text{
  for each } i\}
\end{equation}
for any $\eta \in L^\infty$, for any sequence $\{g_i\}_{i\in \mathbb N}$
in $L^1$ that converges to 0
(see page 235 in \cite{Meg}).

The mapping $\theta_m: L^\infty\ni \xi\mapsto \sup_{i=1,\ldots,m}|\langle \xi-\eta,g_i\rangle |\in \mathbb R$ is 
$\mathcal T_{w\star}$-continuous, hence $\mathcal B(\mathcal T_{w\star})$-measurable\footnote{
  We point out that on the space $\mathbb R$ we always consider the 
  Borel $\sigma$-algebra $\mathcal B^1$. 
  This is not stated at each instance  but tacitely assumed.}.
Therefore,  letting $m\to \infty$ we get that the limit mapping 
$\theta: L^\infty\ni \xi\mapsto \sup_{i\in \mathbb N}|\langle \xi-\eta,g_i\rangle |\in \mathbb R$ is 
 $\mathcal B(\mathcal T_{w\star})$-measurable. This shows that any element \eqref{basis-bw} 
 of the basis   of open subsets with respect to the topology $\mathcal T_{bw\star}$
 belongs to $\mathcal B(\mathcal T_{w\star})$. This implies that 
 $\mathcal B(\mathcal T_{bw\star})\subseteq \mathcal B(\mathcal T_{w\star})$.
\end{proof}

Since in $L^\infty$  the Borelian subsets w.r.t. the weak$\star$
 and the norm topology do not coincide (see \cite{Tal}), we conclude that
\[
\mathcal B(\mathcal T_{w\star})=\mathcal B(\mathcal T_{bw\star})
\subsetneq \mathcal B(\mathcal T_n).
\]
Let us remind that in a separable  Banach space $X$ 
the Borelian subsets w.r.t. the weak
and the norm topology coincide; hence we speak of measurability
meaning that one w.r.t. the (weak=strong) Borelian subsets of $X$.

Finally we deal with the measurability property. Given the mapping
$\omega\in (\Omega,\mathcal F)\to \xi(\omega)\in L^\infty$
we say that it is weakly$\star$ measurable if  for any $g\in L^1$ the mapping
\[
\omega\in \Omega \to \langle \xi(\omega),g\rangle \in \mathbb R
\]
is $\mathcal F\backslash\mathcal B^1$-measurable. 
This is equivalent to say that the mapping $\omega\mapsto \xi(\omega)$
is 
 $\mathcal F\backslash \mathcal B(\mathcal T_{w\star})$-measurable.

\subsection{Existence and uniqueness results}
In this section we collect the basic known results on 
existence and uniqueness for the Euler equation. 
For  $\gamma=0$, these results are stated in a Hilbert setting  in \cite{BessaihF1999,Bes15} and in a more general Banach setting in 
 \cite{BP01}.  The extension to the case $\gamma>0$ is trivial.
We work on any finite time interval $[t_0,T]$; then the results hold on $\mathbb R$.

\begin{theorem}\label{esiste-unico}
Let $\gamma\ge 0$ and assume \eqref{noise} with $h>3$.
\\
i) If $u_0\in V $, then 
on each interval $[t_0,T]$ there exists at least a 
weak global solution for \eqref{Eu} with the initial
condition $u(t_0) = u_0$ satisfying $\mathbf P$-a.s.
$$
 u\in  C([t_0, T]; H) \cap  L^2(t_0, T; V )\qquad 
$$
and, for every $\varphi\in V$ and every $t\in [t_0,T]$ 
\begin{equation*}
 \langle u(t),\varphi\rangle 
 -\int_{t_0}^{t}\langle [u(s)\cdot \nabla] \varphi ,u(s)\rangle ds 
 +\gamma \int_{t_0}^{t}\langle u(s), \varphi\rangle ds 
 =\langle u_{0},\varphi \rangle +\langle W(t)-W(t_0),\varphi \rangle
\end{equation*}
$\mathbf P$-a.s.

Moreover, $u$ is measurable in these topologies and satisfies 
$u \in L^\infty(t_0 ,T; V )$ $\mathbf P$-a.s.
\\
 ii) Let $p \in ]2,\infty[$. 
If $u_0\in V^{1,p}$, 
then the weak global solution $u$ obtained in i)  satisfies 
$$
 u\in  L^\infty(t_0, T; V^{1,p} )\qquad \mathbf P-a.s. 
$$
iii) If $ u_0 \in V$ and $\xi_0=\nabla^\perp\cdot u_0 \in L^\infty$,
then $\xi = \nabla^\perp\cdot u$  (with $u$ 
 the weak global solution obtained in i))  satisfies 
$$
 \xi\in L^\infty([t_0,T]\times D )\qquad \mathbf P-a.s. $$
and  pathwise uniqueness holds.
Moreover $\mathbf P$-a.s.
\[
u \in C_w([0, T]; V), \quad \xi \in C([t_0,T];(L^\infty,\mathcal T_{w\star})),
\]
and  the mapping
\[
(\omega,t)\mapsto \xi(t,\omega)
\]
is jointly measurable, that is 
$\mathcal F\otimes \mathcal B([t_0,T])\backslash \mathcal B(\mathcal T_{w\star})$ measurable. 
\end{theorem}
The important results are about existence; indeed, when the noise is additive
pathwise uniqueness is easily obtained as in the deterministic setting
(see \cite{yudovich,yudovich1995}).

\begin{remark}
a) Here $C_w([0,T];V)$ denotes the space of vectors $u$ 
which are weakly continuous from $[0,T]$ into $V$, i.e. 
for any $\phi \in V^\prime$ the
real mapping
$t \mapsto \langle u(t),\phi\rangle$ is continuous.
\\
b) We say that the mapping
\[
t \in [t_0,T]\mapsto \xi(t) \in  L^\infty
\]
is weakly$\star$  continuous if it is continuous 
when on $ L^\infty$ we consider the weak$\star$ topology
$\mathcal T_{w\star}$. This means that
for any $g \in L^1$ the mapping 
\[
t \in [t_0,T]\mapsto \langle \xi(t),g\rangle \in\mathbb R
\]
is continuous.
\\
c) The measurability of the process $\xi$, 
defined on $(\Omega, \mathcal F, \{ \mathcal F_t\}_t, \mathbf P)$, is obtained by proving that
the mapping 
\[
(\Omega,\mathcal F_t)\ni \omega\to W^{curl}(\cdot )(\omega)\in C((-\infty,t];H^{h-1})
\]
is measurable, and for every $g\in L^1$
 the mapping 
\[
 C((-\infty,t];H^{h-1}) \ni W^{curl} \mapsto \langle \xi(t),g\rangle
 \in \mathbb  R
\]
is continuous.  Hence, composing these two mappings we find that the mapping 
$(\Omega,\mathcal F_t)\ni\omega \mapsto  \langle \xi(t)(\omega), g
\rangle \in \mathbb R$ is measurable, which means that 
$\omega \mapsto\xi(t)(\omega) $
is  $\mathcal F_t\backslash\mathcal B(\mathcal T_{w\star})$  measurable.

Since for  $\mathbf P$-a.e. $\omega \in \Omega$ the mapping
$t \mapsto \langle\xi(t)(\omega),g\rangle$
is  continuous, then the mapping
$(\omega,t)\mapsto\langle \xi(t)(\omega), g\rangle $
is jointly measurable, that is 
the mapping
\[
(\Omega,(\infty,T])\ni (\omega,t)\mapsto  \xi(t)(\omega)\in L^\infty
\]
is 
$\mathcal F_T\otimes \mathcal B((-\infty,T])\backslash \mathcal B(\mathcal T_{w\star})$ measurable. 
\end{remark}

\section{Continuous dependence with respect to the initial data and
  regularity}

The vorticity equation \eqref{vor} can also be rewritten using the Biot-Savart kernel $K$ as follows:
\begin{equation}\label{eq-vor}
 d\xi+ (K*\xi)\cdot\nabla\xi \ dt +\gamma\xi\ dt = dW^{curl}
 \end{equation}

For every $\chi\in L^\infty$, let $\xi(t;\chi)$ be the unique
solution of equation \eqref{eq-vor} evaluated at time $t>t_0$ given the initial
value $\chi$ at time $t_0$. 
By Theorem  \ref{esiste-unico} we have 
$\mathbf P\left(\xi(t;\chi)\in L^\infty\right)=1$. 

Moreover, we can prove a weak form of continuous dependence on the
initial data.
\begin{theorem}\label{dip-cont}
Let $\gamma\ge 0$ and assume \eqref{noise} with $h>3$.
\\
Given a sequence $\left\{\chi^{n}\right\}_n
\subset L^\infty$ which converges weakly$\star$ in $L^\infty$ to
$\chi\in L^\infty$, we have that, $\mathbf P$-a.s., for every $t>t_0$ the sequence 
$\left\{  \xi(t;\chi^{n})\right\}_n  $
converges weakly$\star$ in $L^\infty$ to $\xi(t;\chi)$.
\end{theorem}
\begin{proof} 
In the sequel we work pathwise, that is $\omega $ is fixed in $\Omega $ on a set
  of $\mathbf P$-measure 1. We also fix 
$t_0<T$, and will prove the result for $t\in [t_0,T]$. So all the constants
appearing later depend on $\omega $, $t_0$  and $T$. 

By assumption, we have $\chi^n\in L^\infty$ hence  $\chi^n\in L^p$ for any $p\ge 1$; moreover
$_{L^\infty}\langle \chi^{n},g \rangle_{L^1}\rightarrow 
_{L^\infty}\langle \chi,g \rangle_{L^1}$ for all $g \in L^{1}$.

Set $v^n=u^n-W$,
and $\eta^n=\xi^n-W^{curl}$. Then
\begin{equation}\label{Eu2}
\frac{\partial v^{n}}{\partial t}+\gamma v^n+ B(v^{n}+W, v^{n}+W)=-\gamma W
\end{equation}
and
\begin{equation}\label{rot2}
 \frac{\partial\eta^n}{\partial t}+\gamma \eta^n+ (v^n+W) \cdot\nabla\eta^n= -\gamma W
  -(v^n+W) \cdot\nabla W^{curl} . 
\end{equation}

Since the initial  vorticities are bounded  in $L^\infty$, then the initial velocities are
bounded in $V^{1,p}$ for any finite $p$.
As in Theorem \ref{esiste-unico}, we get that  $\mathbf P$-a.s.
\begin{equation*}\label{tim1}
 \sup_n \sup_{t_0\le t \le T}|v^{n}(t)|^{2}<\infty, \ 
 \sup_n \sup_{t_0\le t \le T}\|v^{n}(t)\|^2<\infty, \ 
  \sup_n \left\|\frac{\partial v^{n}}{\partial t}\right\|_{L^2(t_0,T;V^\prime)}<\infty,
\end{equation*}
and
\begin{equation} \label{stirot}
 \sup_n \sup_{t_0\le t \le T}|\eta ^{n}(t)|^2_{\infty}<\infty.
\end{equation}

From these estimates, following \cite{Temam}, we have that $v^{n}$ is bounded in 
$L^{\infty}(t_0,T;V)\cap H^{1}(t_0,T;V^{\prime })$. So, we can extract a 
subsequence,
still denoted by $\{v^n\}_n$, such that 
$v^n$ converges to some function $v$
strongly in $L^{2}(t_0,T; H)$ and weakly$\star$ in $L^{\infty }(t_0,T;V)$, 
$v^{n}(t)$ converges strongly in $H$ for a.e. $t$, and ${v}$ has the same regularity as
$v^{n}$. Moreover $v\in C([t_0,T];H)$. 

We also deduce that  $\eta ^{n}$ converges to some function $\eta$ weakly$\star$ in 
$L^{\infty }((t_0,T)\times D)$. In particular, for any $g\in L^1$ we have
$\int_{t_0}^T \;_{L^\infty}\langle \eta^n(t), g\rangle_{L^1}  dt
\to \int_{t_0}^T \;_{L^\infty}\langle \eta(t), g\rangle_{L^1} dt$ and  for a.e. $t\in [t_0,T]$
$_{L^\infty}\langle \eta^n(t), g\rangle_{L^1} 
\to  _{L^\infty}\langle \eta(t), g\rangle_{L^1}$.
The same holds for the sequence $\{\xi^n\}_n$, that is for any $g\in L^1$ 
\begin{equation}
\int_{t_0}^T \;_{L^\infty}\langle \xi^n(t), g\rangle_{L^1}  dt
\to \int_{t_0}^T \;_{L^\infty}\langle \xi(t), g\rangle_{L^1} dt
\end{equation}
and 
 for a.e. $t\in [t_0,T]$
\[
 _{L^\infty}\langle \xi^n(t), g\rangle_{L^1} 
\to  _{L^\infty}\langle \xi(t), g\rangle_{L^1}.
\]
Now we show that the limit function $\xi$ is the solution of system \eqref{vor} with initial vorticity $\chi$ and that the convergence holds
for any time $t$.

Let $g \in C^1(D)$; then for a.e.  $t\in [t_0,T] $
\begin{equation}
 \langle \xi^{n}(t),g \rangle +\gamma \int_{t_0}^{t}\langle\xi^n(s),g\rangle ds  
 +\int_{t_0}^{t}\langle u^n(s) \cdot \nabla   \xi^n(s),g \rangle ds
 =\langle \chi^{n},g \rangle +\langle W^{curl}(t)-W^{curl}(t_0),g \rangle.
\end{equation}
Writing
\[
\begin{split}
\int_{t_0}^{t}\langle u^n(s) \cdot &\nabla   \xi^n(s),g \rangle ds
-\int_{t_0}^{t}\langle u(s) \cdot \nabla   \xi(s),g \rangle ds
\\&=-
\int_{t_0}^{t}\langle u^n(s) \cdot \nabla  g, \xi^n(s) \rangle ds
+\int_{t_0}^{t}\langle u(s) \cdot \nabla  g, \xi(s) \rangle ds
\\
&=-\int_{t_0}^{t}\langle [u^n(s) -u(s)]\cdot \nabla  g, \xi^n(s) \rangle ds
-\int_{t_0}^{t}\langle u(s) \cdot \nabla  g,  \xi^n(s)-\xi(s) \rangle ds
\end{split}
\]
and using the strong convergence of $u^n$ and the weak convergence of $\xi^n$, in the limit as $n \to \infty$ we get for any $g \in C^1(D)$ 
for a.e.  $t\in [t_0,T] $
\begin{equation}
 \langle \xi (t),g \rangle +\gamma \int_{t_0}^{t}\langle\xi(s),g\rangle ds  
 -\int_{t_0}^{t}\langle u(s) \cdot \nabla  g, \xi(s) \rangle ds
 =\langle \chi,g \rangle+\langle W^{curl}(t)-W^{curl}(t_0) ,g \rangle .
 \end{equation}
Moreover $t\mapsto \langle \xi (t),g \rangle$ is continuous; 
hence the result holds for any $t \in [t_0,T]$.

Now, by \eqref{stirot}
the sequence $\{\xi^n(t)\}$ and $\xi(t)$ 
are bounded in $L^\infty(D)$; since $C^1(D)$ is dense in $L^1(D)$,
 the Hahn-Banach theorem provides  that for any $t$
\[
 _{L^\infty}\langle \xi^{n}(t)-\xi(t),g\rangle_{L^1}\longrightarrow 0\ \qquad
 \forall g\in L^1(D).
\]
\end{proof}

Now, we state a regularity result on any  finite time interval $[t_0,T]$; 
the state space is now $W^{1,4}(D)$ which is smaller than
$L^\infty(D)$. Hence uniqueness holds true.
The upside of working in $W^{1,4}(D)$ is that this is a separable space, whereas
$L^\infty(D)$ is not. This will be used in the next section. 
The downside is that in  $W^{1,4}(D)$ we are not
able to prove a uniform bound needed for the proof of existence of
invariant measures, whereas we prove it in $L^\infty(D)$
(see Proposition \ref{boundedness}).

\begin{theorem} \label{teo-reg}
Let $\gamma\ge 0$ and assume \eqref{noise} with $h>4$.
\\
If $\xi_{0}\in W^{1,4}(D)$, 
then $\xi\in L^\infty(t_0,T;W^{1,4}(D))\cap
C_w([t_0,T];W^{1,4}(D))$ 
$\mathbf P$-a.s..
\\
Moreover for every $t\in [t_0, T]$, 
the map $(\Omega,\mathcal F_t)\ni \omega\to
\xi(t)(\omega)\in W^{1,4}(D)$ is measurable. 
\end{theorem}
\begin{proof}
We have $W^{1,4}(D)\subset L^\infty(D)$. Hence, 
by the results of Theorem \ref{esiste-unico},  
we only need to prove the estimate for $\nabla \xi$. 
Let us take the gradient of  equation \eqref{vor}: 
\begin{equation}\label{dvor}
d\nabla \xi+\gamma \nabla\xi+\nabla(u\cdot\nabla\xi)\ dt= d\nabla W^{curl}.
\end{equation}
that can be rewritten for each component of the gradient as
\begin{equation}\label{ivor} 
 d \partial_{i} \xi + \gamma \partial_{i} \xi+ \partial_{i}(u\cdot\nabla\xi)\ dt
 = d \partial_i W^{curl}, \qquad i=1,2.
\end{equation}
We look for  $|\nabla\xi| \in L^\infty([t_0,T];L^{4}(D))$. 
Defining $\eta=\xi-W^{curl}$, we get
\begin{equation}\label{eq-z-prime}
 \frac{\partial}{\partial t}\partial_i\eta +\gamma \partial_i\eta  +\partial_i[u\cdot \nabla \eta] = 
 - \partial_i[u \cdot \nabla W^{curl}] -\gamma \partial_iW^{curl},
\qquad i=1,2.
\end{equation}
Let us multiply this equation by $\partial_i  \eta |\nabla \eta|^2$, sum over $i$ and  then integrate over $D$; we get
\begin{equation}\label{L4}
\begin{split}
\frac 14\frac{d}{dt}|\nabla \eta(t)|_4^{4} +\gamma |\nabla \eta(t)|_4^{4} =
&-\sum_{i=1}^2\langle \partial_{i}[u \cdot \nabla \eta] + \partial_i[u
  \cdot \nabla W^{curl}],  \partial_{i}\eta |\nabla \eta|^{2}\rangle\\
  &-\gamma \sum_{i=1}^2\langle \partial_{i} W^{curl},  \partial_{i}\eta |\nabla \eta|^{2}\rangle.
    \end{split}
  \end{equation}
We have
\[
\begin{split}
\sum_i  \langle \partial_{i}[u \cdot \nabla \eta], \partial_{i}\eta |\nabla \eta|^{2}\rangle
&=
\sum_{i,j} \langle \partial_i u_j \partial_j \eta, \partial_i \eta |\nabla \eta|^{2}\rangle + 
\sum_{i,j} \langle u_j \partial^2_{i,j} \eta, \partial_i \eta |\nabla \eta|^{2}\rangle
\\&=:I + II
\end{split}
\]

We use the following result
\begin{lemma}
$II=0$.
\end{lemma}
\begin{proof}
By integration by parts
\begin{equation*}
\begin{split}
II&=\sum_{i,j} \int_{D} u_{j}[\partial_{j}\partial_{i}\eta] 
 \partial_{i}\eta|\nabla \eta|^{2}\\
&=-\sum_i \int_{D} [\sum_j \partial_{j}u_{j}] [\partial_{i}\eta] 
 [\partial_{i}\eta]|\nabla \eta|^{2}
  -\sum_{i,j}\int_{D} u_{j}[\partial_{i}\eta] \partial_{j}
 [\partial_{i}\eta|\nabla \eta|^{2}]\\
&=0- \sum_{i,j}\int_{D} u_{j}[\partial_{i}\eta] [\partial_{j}\partial_i \eta] 
 |\nabla \eta|^{2} -\sum_{i,j}\int_{D} u_{j}[\partial_{i}\eta] 
 [\partial_{i}\eta] \partial_{j}[|\nabla \eta|^{2}]\\
&= -II
  -\sum_{j}\int_{D} u_{j}|\nabla \eta|^{2} \partial_{j}[|\nabla \eta|^{2}]= -3 II
\end{split}
\end{equation*}
Hence, $II=0$. \end{proof}

Now we go back to equation \eqref{L4} and estimate each term in the r.h.s.:
\begin{equation*}
\begin{split}
|I|=|\sum_{i,j} \langle \partial_{i} u_{j}\partial_{j}\eta, \partial_{i}\eta|\nabla \eta|^{2}\rangle| 
&\le C \int_{D}|\nabla u| |\nabla \eta|^{4}
\le C |\nabla u|_{\infty} |\nabla \eta|^{4}_{4}.
 \end{split}
\end{equation*}
Using the H\"older inequality and then the Young inequality, we estimate
another term in \eqref{L4}
\[
\begin{split}
|\sum_{i}\langle \partial_i[u \cdot \nabla W^{curl}], 
   \partial_{i}\eta|\nabla \eta|^{2}\rangle |
&\le
 C \|W^{curl}\|_{V^{2,\infty}} |\nabla u|_{4} ||\nabla \eta|^3|_{4/3} 
\\
&=  
 C \|W^{curl}\|_{V^{2,\infty}} |\nabla u|_{4} |\nabla \eta|_{4}^{3} 
\\
& \le  |\nabla \eta|^4_{4} + C |\nabla u|^4_4  \|W^{curl}\|^4_{V^{2,\infty}}
\end{split}
\]
and we already know that the $ |\nabla u|^4_4$-norm is bounded by Theorem \ref{esiste-unico}.ii.

Similarly for the other term in \eqref{L4} we get

\begin{equation*}
\gamma \sum_{i=1}^2\langle \partial_{i} W^{curl},  \partial_{i}\eta |\nabla \eta|^{2}\rangle
\leq \gamma  |W^{curl}|_{4} |\nabla \eta|_{4}^{3} \leq \frac{\gamma}{2} |\nabla \eta|_{4}^{4} +C(\gamma)  |W^{curl}|_{4}^{4}.
\end{equation*}

 Now, we need an estimate for $ |\nabla u|_{\infty}$. We can find it in 
 Kato \cite{kato}, which deals with the Euler equations in the whole
 plane, or in Ferrari \cite{ferrari},  which deals with the Euler 
equations in a smooth bounded domain of the space; looking at the
proofs of these papers we get for a smooth bounded domain of the plane that
\begin{equation}\label{katostima}
 |\nabla u|_{\infty}\le  C |\xi|_\infty\left[ 1+\log\left(1+\frac{|\nabla
     \xi |_4}{|\xi|_{\infty}}\right)\right].
\end{equation}
Thus, from \eqref{L4} with the above estimates we get
that for any $t \in [0,T]$
\begin{multline}
 \frac 14 \frac{d}{dt}|\nabla \eta(t)|_4^{4}
\le
  |\nabla \eta(t)|_{4}^{4}+
 C|\xi(t)|_\infty\left[1+\log(1+\frac{|\nabla\xi(t)|_4}{|\xi(t)|_{\infty}})\right]
 |\nabla \eta(t)|_4^{4}
\\
 +C(\gamma, t_0,T,|\xi_0|_{\infty},\|W\|_{C([0,\infty);V^{3,\infty})}).
\end{multline}
Gronwall lemma yields 
\begin{equation}
 |\nabla \eta(t)|_{4}^{4}
\le (|\nabla \eta(t_0)|_{4}^{4}+C(t-t_0))
 e^{4\int_{t_0}^t \{1+C|\xi(s)|_\infty
  \left[1+\log(1+\frac{|\nabla\xi(s)|_4}{|\xi(s)|_{\infty}})\right]\}ds}.
\end{equation}
Taking the $\log$ of both sides we get
\begin{multline}
4 \log(|\nabla \eta(t)|_{4}) 
\le
 \log(|\nabla \eta(0)|_{4}^{4}+C(t-t_0))
\\+
4 \int_{t_0}^t \{1+C|\xi(s)|_\infty
  \left[1+\log\Big(1+\frac{|\nabla\xi(s)|_4}{|\xi(s)|_{\infty}}\Big)\right]\}ds.
\end{multline}

Now we use that $\log(x+y) \le \log_+(x+y)\le \log 2 + \log_+ x+\log_+ y$ 
 and 
$- x \log x \le \frac 1e $ (for any $x,y>0$).
Therefore, since $\xi=\eta+W^{curl}$

\begin{multline}
 \log(|\nabla \eta(t)|_{4}) 
\le \frac 14
 \log(|\nabla \eta(t_0)|_{4}^{4}+C(t-t_0))
\\
+C
\int_{t_0}^t \{1+|\xi(s)|_\infty
  \left[C+\log_+(|\nabla\eta(s)|_4)+\log_+(|\nabla W^{curl}(s)|_4)\right]\}ds.
\end{multline}

Using again Gronwall lemma we get
\[
 \sup_{t_0\le t \le T} |\nabla \eta(t)|_{4}
\le
 C(\gamma,t_0, T,|\xi_0|_\infty,\|W\|_{C([t_0,T];V^{3,\infty})}).
\]

Going back to equation \eqref{eq-z-prime} and using the regularity of
$\eta$ obtained so far we get that  ${\bf P}$-a.s.
$\partial_i \eta \in H^1(t_0,T; W^{-1,2}(D))$; 
combining with the fact that  ${\bf P}$-a.s 
$\partial_i \eta\in L^{\infty}(t_0,T; L^4(D))$, then we conclude that 
$\partial_i \eta \in C_w([t_0,T];L^4(D))$, ${\bf P}$-a.s.
(use  Lemma 1.4 of Chapter 3 in  \cite{Temam}).

Finally,  since $\xi=\eta+W^{curl}$ 
and using the regularity of the process $W$ concludes the proof. 

As far the measurability is concerned, this is obtained in a classical
way when working in separable Banach
space see \cite{BessaihF1999}. For a more general theory see e.g. \cite{DPZ}.
\end{proof}

\section{Markov property}
We denote by $ B_b(L^\infty,\mathcal T_{w\star})$ the set of  functions
$\phi:L^\infty\to \mathbb R$ which are bounded  and
$\mathcal B(\mathcal T_{w\star})\setminus\mathcal B^1$-measurable.
Let $\xi(\cdot;\chi)$ be the solution of the vorticity equation 
\eqref{eq-vor} with initial vorticity $\chi$ at time 0.
We define the family of operators (for each $t\ge 0$)
as
\[
(P_t\phi) (\chi)  =\mathbb E\left[\phi(\xi(t;\chi))\right]  
\]
for any $\phi \in B_b(L^\infty,\mathcal T_{w\star})=B_b(L^\infty,\mathcal T_{bw\star})$.

As a  consequence of Theorem \ref{dip-cont} and the Lebesgue dominated convergence
theorem, we infer

\begin{proposition}\label{weak-feller}
The operator  $P_t$  is sequentially weakly$\star$ Feller  in
$L^\infty$  (see \cite{MS}),
that is 
 \begin{equation}\label{swstar}
P_t:SC_b(L^\infty, \mathcal T_{w\star} )\to SC_b(L^\infty, \mathcal T_{w\star} ).
\end{equation}
  \end{proposition}
This is equivalent to say that
\begin{equation}\label{formula-swstar}
P_t: C_b(L^\infty,\mathcal T_{bw\star})\to C_b(L^\infty,\mathcal T_{bw\star}).
\end{equation}
This property will be used in the proof of the main Theorem \ref{th-inv}. Notice that
the bounded weak$\star$ topology  is not metrizable; 
hence, continuity and sequential continuity are different.
So  one proves 
$\mathcal T_{bw\star}$-continuity by means of sequential 
$\mathcal T_{w\star}$-continuity, which is more feasible.

Now we want to show that $P_t$ defines a Markov semigroup.
As far as we know, we have not seen the Markov 
property stated or proved before for the stochastic Euler equations.
This  requires some care since the classical theory for Markov processes
is usually set in Polish spaces.

We proceed in this way. 
First we state an auxiliary result working in the separable Banach space $W^{1,4}(D)$.

\begin{lemma}\label{mark-in-sep}
Let $\gamma\ge 0$ and assume \eqref{noise} with $h>4$.
\\
For every $\phi\in SC_b(L^\infty, \mathcal T_{w\star} )$,  
$\chi \in W^{1,4}(D)$  and $t,s>0$ we have
\begin{equation}\label{Markov-reg}
\mathbf E\left[  \phi\left(  \xi(t+s;\chi)\right)  |{\mathcal F}_{t}\right]
=\left(
P_{s}\phi\right)  \left(  \xi(t;\chi)\right)  \qquad \mathbf P\text{-a.s.}.
\end{equation}
\end{lemma}
\begin{proof}
We divide the proof in four parts. 
For short let $\xi(t; \chi)$ be denoted by $\xi_{t}^{\chi}$; moreover
we use the notation  $\xi_{t,t+s}^{\eta}$ to denote the  solution of
\eqref{eq-vor} (on the  time interval $[t,t+s]$)  
evaluated at time $t+s$ and  started from  $\eta$ at time $t$.

{\bf Step 1}. Given $\phi\in SC_b(L^\infty, \mathcal T_{w\star} )$, 
$\chi\in W^{1,4}(D)$ and $t,s>0$ , \eqref{Markov-reg} is equivalent to 
\begin{equation}\label{Markov-reg2}
\mathbf E\left[  \phi\left(  \xi_{t+s}^{\chi}\right)  Z\right]  
=\mathbf E\left[( P_{s}\phi)  (\xi_{t}^{ \chi})  Z\right]
\end{equation}
for every bounded ${\mathcal F}_{t}$-measurable random variable $Z$. 

Given a $W^{1,4}(D)$-valued ${\mathcal F}_{t}$-measurable random variable 
$\eta$, denote by $\xi_{t,t+s}^{\eta}$ the unique solution of \eqref{vor}
on the time interval $[t,t+s]$ with initial vorticity $\xi(t)=\eta$.
Since by uniqueness
\[
\xi_{t+s}^{\chi}=\xi_{t,t+s}^{\xi_{t}^{\chi}}\quad\text{(}\mathbf P\text{-a.s.)}%
\]
and $\mathbf P\left(  \xi_{t}^{\chi}\in W^{1,4}(D)\right)  =1$ by
Theorem \ref{teo-reg}, 
in order to get \eqref{Markov-reg2} 
it is sufficient to prove that
\begin{equation}\label{mark-2}
\mathbf E\left[  \phi\left(  \xi_{t,t+s}^{\eta}\right)  Z\right]  
=\mathbf E\left[
\left( P_{s}\phi\right)  \left(  \eta\right)  Z\right]
\end{equation}
for every $W^{1,4}(D)$-valued ${\mathcal F}_{t}$-measurable random variable $\eta$.

\textbf{Step 2}. Given such a random variable $\eta$ and since $W^{1,4}(D)$ is a
separable metric space (in contrast to $L^\infty(D)$), there exists
a sequence $\left\{ \eta_{n}\right\}_n  $ of ${\mathcal F}_{t}$-measurable
$W^{1,4}(D)$-valued random variables of the form
\[
\eta_{n}=\sum_{i=1}^{k_{n}}\eta_{n}^{\left(  i\right)
}1_{A_{n}^{(i)}}
\]
with $\eta_{n}^{\left(  i\right)  }\in W^{1,4}(D)$ and $A_{n}^{\left(
i\right) }\in {\mathcal F}_{t}$ with $\{A_n^{(1)}, A_n^{(2)}, \ldots, A_n^{(k_n)}\}$ 
a partition of $\Omega$, 
such that $\left\{  \eta_{n}\right\}  $
converges $\mathbf P$-a.s. strongly in $W^{1,4}(D)$ to $\eta$. If we assume  that
\[
\mathbf E\left[  \phi\left(  \xi_{t,t+s}^{\eta_{n}}\right)  Z\right]
=\mathbf E\left[ \left(  P_{s}\phi\right)  \left(  \eta_{n}\right)
Z\right]\qquad \forall n
\]
then,  since the strong convergence of $\eta_{n}$ in $W^{1,4}(D)$
implies the weak$\star$ convergence in $L^\infty$, using Proposition
\ref{weak-feller} we have that
 $\left(  P_{s}\phi\right) \left(  \eta_{n}\right)  $
converges $\mathbf P$-a.s. to $(P_{s}\phi) (\eta)$. 
On the other side, using Theorem \ref{dip-cont} $\xi_{t,t+s}^{\eta_{n}} $
converges weakly$\star$ in $L^\infty$ to $\xi_{t,t+s}^{\eta}$, so
$\phi\left( \xi_{t,t+s}^{\eta_{n}}\right)  $ also converges
to $\phi\left( \xi_{t,t+s}^{\eta}\right)$ $\mathbf P$-a.s. 
The proof of \eqref{mark-2}  is completed by using the Lebesgue dominated convergence
theorem. 

\textbf{Step 3}. Therefore, it is sufficient to prove \eqref{mark-2}
for every random variable $\eta$ of the form
\[
\eta=\sum_{i=1}^{k}\eta^{\left(  i\right)  }1_{A^{\left(  i\right)  }}%
\]
with $\eta^{\left(  i\right)  }\in W^{1,4}(D)$, $A^{\left(  i\right)
}\in {\mathcal F}_{t} $ and $\{A^{(1)}, A^{(2)}, \ldots, A^{(k)}\}$ 
a partition of $\Omega$. Notice that
\[
\left(  P_{s}\phi\right)  \left(  \eta\right)
=\sum_{i=1}^{k}\left( P_{s}\phi\right)  \left(  \eta^{\left(
i\right)  }\right)  1_{A^{\left( i\right)  }} \qquad \mathbf P-a.s.
\]

Moreover $\xi_{t,t+s}^{\eta}=\sum_{i=1}^{k}\xi_{t,t+s}^{\eta^{\left(  i\right)  }%
}1_{A^{\left(  i\right)  }}$, since we have solved the equation pathwise.
Hence
\[
\phi\left(  \xi_{t,t+s}^{\eta}\right)  =\sum_{i=1}^{k}\phi\left(
\xi_{t,t+s}^{\eta^{\left(  i\right)  }}\right)  1_{A^{\left(  i\right)
}}.
\]

Thus it is sufficient to prove
\begin{equation}\label{mark-3}
\mathbf E\left[  \phi\left(  \xi_{t,t+s}^{\eta^{\left(  i\right)  }}\right)
1_{A^{\left(  i\right)  }}Z\right]  =\mathbf E\left[  \left(
P_{s}\phi\right) \left(  \eta^{\left(  i\right)  }\right)
1_{A^{\left(  i\right)  }}Z\right]
\end{equation}
for every $i$.

\textbf{Step 4}. Since $1_{A^{\left(  i\right)  }}$ is a bounded 
${\mathcal F}_{t}$-measurable random variable, 
in order to prove \eqref{mark-3} it is sufficient that
\[
\mathbf E\left[  \phi\left(  \xi_{t,t+s}^{\eta}\right)  Z\right]  
=\mathbf E\left[
\left( P_{s}\phi\right)  \left(  \eta\right)  Z\right]
\]
for every bounded ${\mathcal F}_{t}$-measurable random variable $Z$ and every
deterministic element $\eta\in W^{1,4}(D)$. The random variable $\xi_{t,t+s}^{\eta}$
depends only on the increments of the  Wiener process between $t$
and $t+s$, hence it is independent of ${\mathcal F}_{t}$. Therefore
\[\begin{split}
\mathbf E\left[  \phi\left(  \xi_{t,t+s}^{\eta}\right)  Z\right]  
&=
\mathbf E\left[  \mathbf E[\phi\left(  \xi_{t,t+s}^{\eta}\right)  Z|{\mathcal F}_t]\right] 
=
\mathbf E\left[Z\  \mathbf E[\phi\left(  \xi_{t,t+s}^{\eta}\right)  |{\mathcal F}_t]\right] 
\\
&=
\mathbf E\left[Z\  \mathbf E[\phi\left(  \xi_{t,t+s}^{\eta}\right)]\right]
=\mathbf E\left[
\phi\left(  \xi_{t,t+s}^{\eta}\right)  \right]  \mathbf E\left[  Z\right]
.
\end{split}
\]
Since $\xi_{t,t+s}^{\eta}$ and  $\xi_{s}^{\eta}$ have the same law, 
we have $\mathbf E\left[  \phi\left(
\xi_{t,t+s}^{\eta}\right)  \right]  =\mathbf E\left[ \phi\left(
\xi_{s}^{\eta}\right)  \right]  $ and thus we have proved that 
\[
\mathbf E\left[  \phi\left(  \xi_{t,t+s}^{\eta}\right)  Z\right]  
=\mathbf E\left[  \phi\left(  \xi_{s}^{\eta}\right) \right] \ \mathbf E\left[ Z\right] 
=\left(
P_{s}\phi\right)  \left(  \eta\right)  \mathbf E\left[  Z\right]
=\mathbf E\left[  \left( P_{s}\phi\right)  \left(  \eta\right)  Z\right]
.
\]
The proof is complete.
\end{proof}

Now, we are ready to state the main result related to the Markov property.
The following proposition is one possible Markov property 
for the family of solutions to equation \eqref{eq-vor}.
\begin{proposition}\label{proMARK}
Let $\gamma\ge 0$ and assume \eqref{noise} with $h>4$.
\\
For every $\phi\in SC_b(L^\infty,\mathcal T_{w\star})$,  
$\chi \in L^\infty$ and $t,s>0$,  we have
\begin{equation}\label{markk}
 \mathbf E\left[  \phi\left(  \xi(t+s;\chi)\right)  |{\mathcal F}_{t}\right]
 =\left( P_{s}\phi\right)  \left(  \xi(t;\chi) \right)  \qquad \mathbf P-a.s..
\end{equation}
\end{proposition}

\begin{proof}
The space $W^{1,4}(D)$ is densely embedded in the
space $(L^\infty,\mathcal T_{w\star})$, see \cite{brezis}. By the way,
this shows that  $(L^\infty,\mathcal T_{w\star})$ is a separable space.

 Thus, given $\chi \in L^\infty$ there is a
sequence $\left\{  \chi^{n}\right\}  \subset W^{1,4}(D)$ which
converges weakly$\star$ in $L^\infty$ to $\chi$. Lemma 
\ref{mark-in-sep} infers 
that, given $\phi\in SC_b(L^\infty, \mathcal T_{w\star} )$ and  $t,s>0$,
\[
 \mathbf E\left[  \phi\left(  \xi(t+s;\chi^{n})\right)  |{\mathcal F}_{t}\right]
 =\left(
 P_{s}\phi\right)  \left(  \xi(t;\chi^{n})\right)  \qquad \mathbf P-a.s.
\]
This means that
\[
\mathbf E\left[  \phi\left(  \xi(t+s;\chi^{n})\right)  Z\right]
=\mathbf E\left[ \left(  P_{s}\phi\right)  \left(\xi(t;\chi^{n})\right)  Z\right]
\]
for every bounded ${\mathcal F}_{t}$-measurable random variable $Z$.

From Theorem \ref{dip-cont} we know that for any $r>0$ 
$\left\{  \xi(r;\chi^{n})\right\}  $ converges
weakly$\star$ in $L^\infty$ to $\xi(r;\chi)$, $\mathbf P$-a.s.. 
Hence $\left(  P_{s}\phi\right)  \left(
\xi(t;\chi^{n})\right)  $ converges to $\left(
P_{s}\phi\right)  \left(  \xi(t;\chi)\right)  $, $\mathbf P$-a.s., and
$\phi\left(  \xi(t+s;\chi^{n})\right)  $ converges to
$\phi\left( \xi(t+s;\chi)\right)  $, $\mathbf P$-a.s., and thus by
Lebesgue dominated convergence theorem we can pass to the limit in
the previous equation and get
\[
\mathbf E\left[  \phi\left(  \xi(t+s;\chi)\right)  Z\right]  
=\mathbf E\left[
\left( P_{s}\phi\right)  \left(  \xi(t;\chi)\right)  Z\right]  \qquad \mathbf P-a.s.\]
This is equivalent to \eqref{markk}.
\end{proof}

\begin{corollary}
For any $s,t\ge 0$ we have
$P_{t+s}=P_{t}P_{s}$ on $ SC_b(L^\infty, \mathcal T_{w\star} )$.
\end{corollary}
\begin{proof}
Taking the expectation in \eqref{markk}, we have
\[
\mathbf E\left[  \phi\left(  \xi(t+s;\chi)\right)  \right]  
=\mathbf E\left[
\left( P_{s}\phi\right)  \left(  \xi(t;\chi)\right)  \right]
\]
which  can be rewritten as
$(P_{t+s} \phi)(\chi)=\left( P_{t}(  P_{s}\phi)  \right) (\chi)$.
\end{proof}

\section{Invariant measures}
Let us consider the Markov semigroup $\{P_t\}_{t\ge 0}$ acting in 
$SC_b(L^\infty,\mathcal T_{w\star})=C_b(L^\infty,\mathcal T_{bw\star})$, associated to the equation \eqref{eq-vor}.
We say that a probability measure $\mu$  on $\mathcal B(\mathcal T_{bw\star})$ is an invariant measure for it if
\begin{equation}\label{def-inv}
  \int P_t \phi\ d\mu=\int \phi\ d\mu
  \qquad \forall t\ge 0, \forall \phi \in C_b(L^\infty,\mathcal T_{bw\star})
\end{equation}
We want to prove existence of an invariant measure 
by means of Krylov-Bogoliubov's method. We recall that already  
Maslowski and Seidler in \cite{MS} used this method with weak
topologies, 
but assuming that the state space is a
separable Hilbert space. Anyway also when dealing with the space
$L^\infty$ we can proceed along the lines of
Krylov-Bogoliubov's method in order to prove existence of invariant
measures defined on $\mathcal B(\mathcal T_{bw\star})$.

This is our result
\begin{theorem}\label{th-inv}
Let $\gamma> 0$ and assume \eqref{noise} with $h>4$.
\\
Then there exists at least one invariant measure 
for the stochastic equation \eqref{eq-vor}.
\end{theorem}
\begin{proof}
The idea is to construct a  sequence  of measures
$\{\mu_n\}_{n \in \mathbb N}$, which  is $\mathcal T_{bw\star}$-tight;
from it we can extract a subsequence converging to a measure $\mu$;
then we show that  the limit measure
$\mu$ is an invariant measure, thanks to \eqref{formula-swstar}.

We denote by $m_t$ the law of the random variable $\xi(t;0)$ on $\mathcal B(\mathcal T_{bw\star})$; 
 since the mapping $(\omega,t)\mapsto \xi(t;x)(\omega)$
is jointly measurable, we can integrate with respect to both variables 
and define the  probability measure on $\mathcal B(\mathcal T_{bw\star})$
  \[
  \mu_n=\frac 1n \int_0^n m_t\ dt
  \]
for any $n>0$.

We recall that the set $\{\|x\|_{L^\infty}\le R\}$ is $\mathcal T_{bw\star}$-compact.
From Corollary \ref{c-b}, which will be proved in the next subsection,
we have that the sequence 
 $\{\mu_n\}_{n \in \mathbb N}$   is $\mathcal T_{bw\star}$-tight, that is
  \[
  \forall \epsilon >0 \ \exists K_\epsilon \ \mathcal T_{bw\star}\text{-compact subset of } L^\infty:
  \inf_n \mu_n(K_\epsilon)>1-\epsilon. 
\]

Now we apply  Prokhorov's theorem in the version given by Jakubowski 
(see Theorem 3 in \cite{Ja}), which allows
 to work in non metric spaces. This requires that the space
 $L^\infty$ with the bounded weak$\star$ topology 
$\mathcal T_{bw\star}$ is countably separated, that is
there exists a countable family $\{g_i:L^\infty\to [-1,1]\}_{i \in
  \mathbb N}$
of $\mathcal T_{bw\star}$-continuous functions which separate points of $L^\infty$.
This is our case, since $L^1$ is separable, so there exists a
countable sequence $\{h_i\}_i \subset L^1$ separating the points of
$L^\infty$, that is for any two elements $x\neq y $ in $L^\infty$
there exists $h_i$ such that 
$\langle x,h_i\rangle \neq \langle y,h_i\rangle$.
Since the mapping $x\mapsto \langle x,h_i\rangle$ is 
$\mathcal T_{w\star}$-continuous, then it is also 
$\mathcal T_{bw\star}$-continuous.

Therefore there exists a subsequence $\{\mu_{n_k}\}_k$ and a probability
measure $\mu$  on $\mathcal B(\mathcal T_{bw\star})$ such that $\mu_{n_k}$
converges narrowly to $\mu$ as $k \to \infty$ ($n_k \to \infty$), that is
\[
  \int\phi\ d\mu_{n_k} \to \int\phi\ d\mu
  \qquad\qquad \forall \phi \in C_b(L^\infty,\mathcal T_{bw\star}).
\]

On the other hand we have that
\[
  \langle P_t\phi,\mu_{n_k}\rangle = \langle \phi,\mu_{n_k}\rangle
  +\frac 1{n_k}\int_{n_k}^{t+n_k}\langle \phi, m_u \rangle du
  -\frac 1{n_k}\int_{0}^{t}\langle \phi,m_u \rangle du.
\]
Letting $k \to \infty$, the two latter terms vanish.
From  \eqref{formula-swstar} we know that $P_t\phi\in C_b(L^\infty,\mathcal T_{bw\star})$
if $\phi\in C_b(L^\infty,\mathcal T_{bw\star})$. Hence in the limit we obtain
\[
  \langle P_t\phi,\mu\rangle= \langle \phi,\mu\rangle 
  \]
for each $\phi\in C_b(L^\infty,\mathcal T_{bw\star})$ and each $t\ge 0$.
\end{proof}

\begin{remark}
Maslowski and Seidler in \cite{MS} 
proved existence of an invariant measure dealing with weak topologies; applications can be found in 
that work and also in some papers by 
Brze\`zniak and collaborators, see \cite{BOS,BMO,BFe}; in all these works the state space is a separable Hilbert space.
Working with the weak topologies  is an improvement in applications, since it is easier to prove the
tightness with respect to weak topologies than  with respect to the
strong ones.
For  instance we prove the weak tightness for the damped Euler
equation \eqref{eq-vor} whereas the tightness with respect to the
strong topology requires a dissipative term of the form $-\Delta \xi$
(or a fractional power of the Laplacian operator), that is it
holds for
the Navier-Stokes equations or fractional Navier-Stokes equations but
not for the Euler equations.

The classical Krylov-Bogoliubov's method 
 is based on the tightness and the Feller property (see, e.g.,  \cite{DPZ,DPZ2}).
Therefore Maslowski and Seidler realized that dealing with weak
topologies for the tightness called for a ''weak'' Feller property too.
Actually, working in a separable Hilbert space
$H$ they considered the  weak 
topology $\mathcal T_w$ and the strong topology $\mathcal T_n$, and 
proved the existence of an invariant measure
by assuming
\begin{enumerate}
  \item
$P_t:SC_b(H,\mathcal T_{w})\to SC_b(H,\mathcal T_{w})$
\item
  the family $\{\mu_n\}_{n \in \mathbb N}$   is $\mathcal T_{w}$-tight
  \end{enumerate}

Let us point out that 
taking into account the bounded weak topology $\mathcal T_{bw}$
(which they considered in a subsequent paper \cite{MS01}), one can
write the two assumptions in an equivalent way as
\begin{enumerate}
  \item
$P_t:C_b(H,\mathcal T_{bw})\to C_b(H,\mathcal T_{bw})$
\item
  the family $\{\mu_n\}_{n \in \mathbb N}$   is $\mathcal T_{bw}$-tight
  \end{enumerate}
since
$SC(H,\mathcal T_{w})=C(H,\mathcal T_{bw})$ and 
$\mathcal T_{w}$-compact subsets coincide with the 
$\mathcal T_{bw}$-compact subsets of $H$.
This simplifies a bit  the proof of Theorem 3.1 in \cite{MS}, looking
more similar to that of the classical   Krylov-Bogoliubov's
theorem.
So in principle the weak topology $\mathcal T_{w}$ does not appear in
the assumptions. However, one proves 
$\mathcal T_{bw}$-continuity by means of sequential 
$\mathcal T_{w}$-continuity, which is easier to prove.
\end{remark}

\subsection{Boundedness in probability}
Here we prove the uniform bound in probability needed in the last proof.

\begin{proposition}\label{boundedness}
Let $\gamma>0$ and assume  \eqref{noise} with $ h>3$.
\\
Then, there exists a real random variable $r$ ($\mathbf P$-a.s.
 finite) such that
\begin{equation}\label{bound-probability}
\sup_{t_0\le 0} \left|\xi(0;\xi(t_0)=0)\right|_{\infty}\le r
\qquad \mathbf P-a.s.
\end{equation}
\end{proposition}
\begin{proof}
 Our proof will follow a similar result introduced by Flandoli in \cite{Fla94} which uses dissipative features of the Navier-Stokes equations and  the ergodic properties of an auxiliary Ornstein-Uhlenbeck process. 

We introduce the linear equation
\begin{equation}\label{OU-curl}
          d\zeta_\lambda(t)+\lambda \zeta_\lambda(t)dt= dW^{curl}(t)
\end{equation}
for $\lambda>0$; its  stationary solution is
\begin{equation}\label{curlZ}
\zeta_\lambda (t)=\int_{-\infty}^{t}e^{-\lambda(t-s)}dW^{curl}(s).
\end{equation} 
Set $\eta_\lambda=\xi-\zeta_\lambda$. 
Then $\eta_\lambda$ fulfils the following equation
\begin{equation}\label{rot2-nu-}
 \frac{\partial\eta_\lambda}{\partial t}+\gamma \eta_\lambda
 +[K\star(\eta_\lambda+\zeta_\lambda)] \cdot\nabla\eta_\lambda= 
 -[K\star(\eta_\lambda+\zeta_\lambda)] \cdot\nabla \zeta_\lambda +(\lambda-\gamma)\zeta_\lambda. 
\end{equation}
We multiply equation \eqref{rot2-nu-} by 
$|\eta_\lambda|^{p-2}\eta_\lambda, \quad p\geq 2$, and integrate over the spatial domain $D$; using that
$\langle u\cdot \nabla\eta_\lambda, |\eta_\lambda| ^{p-2}\eta_\lambda
\rangle=0$ (here the estimates are first performed on more regular
solutions, the Navier-Stokes approximations, and then pass to the
limit for vanishing  viscosity), we infer that
\[
\begin{split}
\frac{1}{p}\frac{d}{dt} & \left| \eta_\lambda(t)\right|_{p}^{p} + \gamma \left| \eta_\lambda(t)\right|_{p}^{p}
=- \langle [K\star(\eta_\lambda(t)+\zeta_\lambda(t))]\cdot \nabla\eta_\lambda(t), |\eta_\lambda(t)| ^{p-2}\eta_\lambda(t) \rangle\\ 
& \quad - \langle [K\star(\eta_\lambda(t)+\zeta_\lambda(t))]\cdot  \nabla \zeta_\lambda(t),  |\eta_\lambda(t)| ^{p-2}\eta_\lambda(t) \rangle
 + (\lambda-\gamma)\langle\zeta_\lambda(t),  |\eta_\lambda(t)| ^{p-2}\eta_\lambda(t)\rangle\\
& \le   |K\star(\eta_\lambda(t)+\zeta_\lambda(t))|_{p}  |\nabla \zeta_\lambda(t)|_{\infty} | \eta_\lambda(t)|_{p}^{p-1}
+ |\lambda-\gamma||\zeta_\lambda(t)|_p |\eta_\lambda(t)|_p^{p-1}\\
& \le \Big[ C( | \eta_\lambda(t)|_{p}+|\zeta_\lambda(t)|_{p} ) |\nabla \zeta_\lambda(t)|_{\infty}+ |\lambda-\gamma||\zeta_\lambda(t)|_p\Big] | \eta_\lambda(t)|_{p}^{p-1}.
   \end{split}
\]

On the other side, we have that 
$\frac{d}{dt} \left| \eta_\lambda(t) \right|_{p} ^{p}=p\left| \eta_\lambda(t) \right|_{p}^{p-1}
 \frac{d}{dt} \left| \eta_\lambda(t) \right|_{p}$; we deduce that for any arbitrary 
 $p\ge 1$
\begin{align*}
\frac{d}{dt} \left| \eta_\lambda(t)\right|_{p}
+ \gamma\left| \eta_\lambda(t)\right|_{p}
 \le  C( | \eta_\lambda(t)|_{p}+|\zeta_\lambda(t)|_{p} ) |\nabla
 \zeta_\lambda(t)|_{\infty}+ |\lambda-\gamma| |\zeta_\lambda(t)|_p .
 \end{align*} 
Hence
\begin{align*}
\frac{d}{dt} \left| \eta_\lambda(t)\right|_{p}
+ \big(\gamma - C|\nabla \zeta_\lambda(t)|_{\infty} \big) \left| \eta_\lambda(t)\right|_{p}
\le 
 \big(C |\nabla \zeta_\lambda(t)|_{\infty}+ |\lambda-\gamma|\big) |\zeta_\lambda(t)|_p .
 \end{align*} 
Now Gronwall's inequality yields on the interval $[t_0,0]$
\begin{multline}\label{p}
 |\eta_\lambda(0)|_p  
 \le 
 |\eta_\lambda(t_0)|_p e^{-\int_{t_0}^{0} (\gamma - C|\nabla \zeta_\lambda(s)|_{\infty} )ds}
 \\+ \int_{t_0}^0  \big(C |\nabla \zeta_\lambda(s)|_{\infty}+ |\lambda-\gamma|\big) |\zeta_\lambda(s)|_p\ 
 e^{-\int_{s}^{0} (\gamma -  C|\nabla \zeta_\lambda(r)|_{\infty} )dr} ds
 \end{multline}
Using that $H^{a-1}\subset L^\infty$ for any $a>2$ and taking $p\to\infty$, we get that
\begin{multline*}
 |\eta_\lambda(0)|_\infty 
 \le 
 |\eta_\lambda(t_0)|_\infty e^{-\int_{t_0}^{0} (\gamma - \tilde C\| \zeta_\lambda(s)\|_{H^{a}} )ds}
 \\+ \int_{t_0}^0  C\big(\|\zeta_\lambda(s)\|_{H^a}+ |\lambda-\gamma|\big) \|\zeta_\lambda(s)\|_{H^a} 
 e^{-\int_{s}^{0} (\gamma - \tilde C \|\zeta_\lambda(r)\|_{H^a} )dr} ds
 \end{multline*}
for some positive constants $C$ and $\tilde C$.
Since $\xi(t_0)=0$, we have
\begin{multline}   \label{infty}
 |\eta_\lambda(0)|_\infty 
 \le 
C \|\zeta_\lambda(t_0)\|_{H^a} e^{-\int_{t_0}^{0} (\gamma - \tilde C\| \zeta_\lambda(s)\|_{H^{a}} )ds}
 \\+ \int_{t_0}^0  C\big(\|\zeta_\lambda(s)\|_{H^a}+ |\lambda-\gamma|\big) \|\zeta_\lambda(s)\|_{H^a} 
 e^{-\int_{s}^{0} (\gamma - \tilde C \|\zeta_\lambda(r)\|_{H^a} )dr} ds
 \end{multline}
 
Now we choose $\lambda$ large enough in order to have a uniform bound. First of all we require that 
$\int_{t_0}^{0} (\gamma - \tilde C\| \zeta_\lambda(s)\|_{H^{a}} )ds>0$. 
To this end, we notice that the process $\zeta_\lambda$ has the same regularity 
as $W^{curl}$ and using that 
$\mathbf E\left[ \int_{-\infty}^t e^{-\lambda(t-s)} d\beta_i(s) \int_{-\infty}^t e^{-\lambda(t-r)} d\beta_j(r) \right]=
\delta_{ij}\frac 1{2\lambda}$,
we compute
\[
\mathbf E \left[\|\zeta_\lambda(t)\|^2_{H^a}\right]
=\frac 1{2\lambda} \mathbf E\left[ \|W^{curl}(1)\|^2_{H^a}\right].
\]
Since $\zeta_\lambda$ is an ergodic process (see, e.g., \cite{DPZ2}) we have
\begin{equation*}
\lim_{t_0\to -\infty} \frac{1}{-t_0}
\int_{t_0}^{0} \|\zeta_\lambda(s)\|_{H^a} ds=\mathbf E \|\zeta_\lambda(0)\|_{H^a} \qquad \mathbf P-a.s.
\end{equation*}
We choose $\lambda$ large enough such that 
\begin{equation}\label{Gamma1}
\tilde C \mathbf E \|\zeta_\lambda(0)\|_{H^a}\le
\frac {\tilde C}{\sqrt{2\lambda}} \sqrt{\mathbf E \|W^{curl}(1)\|^2_{H^a}}
<\frac \gamma 2
\end{equation}
where $\tilde C$ is the constant appearing in \eqref{infty}; thus
\[
\lim_{t_0\to -\infty} \frac{1}{-t_0}\int_{t_0}^{0} \tilde C \|\zeta_\lambda(s)\|_{H^a} ds<\frac \gamma 2
 \qquad \mathbf P-a.s.
\]
Then, given $\omega\in\Omega$ there exists $\tau(\omega)<0$ such that 
\begin{equation}
\int_{t_0}^{0} \tilde C \|\zeta_\lambda(s)\|_{H^a} ds\le \frac{\gamma}{2} (-t_0),\quad \forall t_0<\tau(\omega).
\end{equation}
Moreover, by the continuity of the trajectories of $\zeta_\lambda$,
there exists a (random) constant $r_1$, $\mathbf P$-a.s. finite,
such that
\begin{equation*}
\displaystyle \sup_{\tau(\omega)<t_0\le 0}\int_{t_0}^{0} \tilde C \|\zeta_\lambda(s)\|_{H^a} ds\le r_1
\end{equation*}
$\mathbf P$-a.s..
Hence
\[
e^{-\int_{t_0}^{0} (\gamma - \tilde C\| \zeta_\lambda(s)\|_{H^{a}} )ds}
\]
is (pathwise) uniformly bounded for $t_0<0$ and vanishes exponentially fast as $t_0\to -\infty$.

Now, arguing as before we get that there exists a 
a random variable $r_2$ ($\mathbf P$-a.s. finite) such that 
$\mathbf P$-a.s. we have
\begin{equation}\label{z2}
\|\zeta_\lambda(t)\|_{H^a} \le r_2 (|t|+1)  \quad  t<0.
\end{equation}

Thus we have proved 
a uniform bound for each term  in the r.h.s. of estimate \eqref{infty}, that is we have obtained that
 there exists a random variable $r_3$ ($\mathbf P$-a.s. finite) such that
 \[
\sup_{t_0\le 0} |\eta_\lambda(0;\eta(t_0)=-\zeta_\lambda(t_0))|_\infty \le r_3 \qquad \mathbf P-a.s.
 \]
 Since $\xi=\eta_\lambda+\zeta_\lambda$ , we obtain \eqref{bound-probability}.
\end{proof}

From this we get
\begin{corollary}\label{c-b}
Let $\gamma>0$ and assume  \eqref{noise} with $ h>3$.
\\
Then, for any $\epsilon>0$ there exists $R_\epsilon>0$ such that
\[
\inf_{t\ge 0}\mathbf P\{ \left|\xi(t;\xi(0)=0)\right|_{\infty} \le R_\epsilon\}\ge 1-\epsilon .
\]
\end{corollary}
\begin{proof}
First, let us note that for any $t_0<0$ the random variables 
$\xi(0;\xi(t_0)=0)$ and $\xi(-t_0;\xi(0)=0)$
have the same law (homogeneity). 
Moreover, given a random varible $r$ which is non negative and finite, we have that 
for any $\epsilon>0$ there exists $R_\epsilon>0$ such that 
\[
\mathbf P\{ r\le R_\epsilon\}\ge 1-\epsilon.
\]
Therefore, keeping in mind the result of Proposition \ref{boundedness} we get
\[
\mathbf P\{ \left|\xi(t;\xi(0)=0)\right|_{\infty} \le R_\epsilon\}
=
\mathbf P\{ \left|\xi(0;\xi(-t)=0)\right|_{\infty} \le R_\epsilon\}
\ge \mathbf P\{ r \le R_\epsilon\}\ge 1-\epsilon
\]
and this estimate is uniform in time.
\end{proof}

\medskip
{\bf Acknowledgements.}
H. Bessaih was partially supported by Simons Foundation grant 582264
and by INdAM-GNAMPA to visit the Department of Pavia. 
\\
B. Ferrario was  partially supported by INdAM-GNAMPA,  
by MIUR-Di\-par\-ti\-menti di Eccellenza Program (2018-2022) and by
PRIN 2015 ''Deterministic and stochastic evolution equations''.


\begin{thebibliography}{11}

\bibitem{Bardos}
C. Bardos:
Existence et unicit\'{e} de la solution de l'\'{e}quation d'Euler en dimensions deux, 
{\it Jour. Math. Anal. Appl.} {\bf 40} (1972),  769--780. 

\bibitem{Bessaih2000}
H. Bessaih:
Stochastic weak attractor for a dissipative Euler equation, 
{\it Electron. J. Probab.} {\bf 5} (2000), no. 3, 16 pp. 

\bibitem{Bessaih2008}
H. Bessaih:
Stationary solutions for the 2D stochastic dissipative Euler equation, 
{\it Seminar on Stochastic Analysis, Random Fields and Applications V}, 23--36, 
Progr. Probab. 59, Birkh\"auser, Basel, 2008. 

\bibitem{Bes15}
H. Bessaih: 
Stochastic incompressible Euler equations in a two-dimensional domain,
{\it Stochastic analysis: a series of lectures},  135--155, 
Progr. Probab., 68, Birkh\"user/Springer, Basel, 2015.

\bibitem{BesFer13}
H. Bessaih, B. Ferrario:
Inviscid limit of stochastic damped 2D Navier-Stokes equations, 
{\it Nonlinearity} {\bf 27} (2014), no.1, 1--15.

\bibitem{BessaihF1999}
H. Bessaih, F. Flandoli:
2-D Euler equation perturbed by noise, 
{\it NoDEA Nonlinear Differential Equations Appl.} {\bf 6} (1999), no. 1, 35--54.

\bibitem{BF}
H. Bessaih, F. Flandoli:
Weak attractor for a dissipative Euler equation, 
{\it J. Dynam. Differential Equations} {\bf 12} (2000), no. 4, 713--732.

\bibitem{boffetta}
G. Boffetta, R. Ecke: 
Two-dimensional turbulence, 
{\it Annu. Rev. Fluid Mech.} {\bf 44} (2012), 427--451.
 
\bibitem{brezis}
H. Br\'ezis: 
{\it Functional analysis, Sobolev spaces and partial differential equations}.
Universitext. Springer, New York, 2011. 

\bibitem{BFe}
Z. Brze\'zniak, B. Ferrario:
Stationary solutions for stochastic damped Navier-Stokes equations in $R^d$,
 {\it Indiana Univ. Math. J.} {\bf 68} (2019), No. 1, 105--138.

\bibitem{BMO}
Z. Brze\'zniak, E. Motyl, M. Ondrej\`at:
Invariant measure for the stochastic Navier-Stokes equations in unbounded 2D domains,
{\it Annals Probab.} {\bf 45}  (2017),  No. 5, 3145--3201.

\bibitem{BOS}
Z. Brze\'zniak, M. Ondrej\`at, J. Seidler:
Invariant measures for stochastic nonlinear beam and wave equations,
{\it J. Diff. Equations} {\bf 260} (2016), No. 5, 4157--4179.

\bibitem {BP01}
Z. Brze\'zniak, S. Peszat: 
Stochastic two dimensional Euler equations, 
{\it Annal.  Probab.}, {\bf 29} (2001), no. 4, 1796--1832.

\bibitem{CGHV}
P. Constantin, N. E. Glatt-Holtz, V. C. Vicol: 
Unique ergodicity for fractionally dissipated, stochastically forced
2D Euler equations,
{\it Comm. Math. Phys.} {\bf 330} (2014), no. 2, 819--857.

\bibitem{DPZ}
G. Da Prato, J. Zabczyk: 
{\it Stochastic Equations in Infinite Dimensions}.
Encyclopedia of Mathematics and its Applications 44, Cambridge University
Press, 1992.

\bibitem{DPZ2}
G. Da Prato, J. Zabczyk: 
{\it Ergodicity for Infinite Dimensional Systems}.
London Mathematical Socienty Lecture Note Series 229, Cambridge
University Press, 1996.

\bibitem{ferrari}
A. B. Ferrari:
On the blow-up of solutions of the 3-D Euler equations in a bounded
domain,
{\it Comm. Math. Phys.} {\bf 155} (1993), no. 2, 277--294. 

\bibitem{Fla94}
F. Flandoli: 
Dissipativity and invariant measures for stochastic Navier-Stokes
equations,
{\it NoDEA Nonlinear Differential Equations Appl.} 1 (1994), no. 4, 403-423

\bibitem{gal}
G. Gallavotti:
{\it Foundations of fluid dynamics}.
 Springer-Verlag, Berlin, 2002.

\bibitem{HM}
  M. Hairer, J.C. Mattingly:
  Ergodicity of the 2D Navier-Stokes equations with degenerate stochastic forcing,
  {\it Ann. of Math.}  {\bf 164} (2006), no. 3, 993--1032. 

\bibitem{Ja}
A. Jakubowski:
The almost sure Skorokhod representation for subsequences in nonmetric
spaces,
{\it Teor. Veroyatnost. i Primenen.} {\bf 42}  (1997), no.~1, 209-216;
translation in {\it Theory  Probab. Appl.} {\bf 42}  (1998), no.~1, 167-174

\bibitem{kato}
T. Kato:
 Remarks on the Euler and Navier-Stokes equations in $R^2$,
{\it Nonlinear functional analysis and its applications}, 
Part 2 (Berkeley, Calif., 1983), 17, Proc. Sympos. Pure Math., 45, 
Part 2, Amer. Math. Soc., Providence, RI, 1986.

\bibitem{Ku}
A. Kupiainen:
Ergodicity of two dimensional turbulence (after Hairer and Mattingly).
S\'eminaire Bourbaki. Vol. 2009/2010. Expos\'es 1012-1026. 
{\it Ast\'erisque} {\bf 339} (2011), Exp. No. 1016, vii, 137--156.

\bibitem{MS}
B. Maslowski, J. Seidler:
On sequentially weakly Feller solutions to SPDE's.
{\it Atti Accad. Naz. Lincei} Cl. Sci. Fis. Mat. Natur. Rend. Lincei (9)
Mat. Appl. 10 (1999), no. 2, 69-78

\bibitem{MS01}
B. Maslowski, J. Seidler:
Strong Feller solutions to SPDE's are strong Feller in the weak
topology. 
{\it Studia Math.} 148 (2001), no. 2, 111-129. 

\bibitem{Meg}
R. E. Megginson: {\it An introduction to Banach space theory}, 
 Graduate Texts in Mathematics, 183. Springer-Verlag, New York, 1998.

\bibitem{Tal}
M. Talagrand:
Comparaison des Bor\'eliens pour les topologies fortes et faibles.
{\it  Indiana Univ. Math. J.} {\bf 21} (1978), 1001--1004.

\bibitem{Temam}
R. Temam:  
{\it Navier-Stokes equations. Theory and numerical analysis.} 
Revised edition, 
Studies in Mathematics and its Applications, 2. North-Holland
Publishing Co., Amsterdam-New York  (1979).

\bibitem{yudovich}
V. I. Yudovic: 
Non-stationary flows of an ideal incompressible fluid.  (Russian) 
{\it  \v{Z}.  Vy\v{c}isl. Mat. i Mat. Fiz. }
{\bf 3} (1963), 1032--1066. Translation in English in 
{\it  U.S.S.R. Comput. Math. and Math. Phys. }
{\bf 3} (1963), no. 6, 1407--1456.

\bibitem{yudovich1995}
V. I. Yudovich: 
Uniqueness theorem for the basic nonstationary problem in the dynamics
of an ideal incompressible fluid. 
{\it Math. Res. Lett.} {\bf 2} (1995), 27--38. 


\end{thebibliography}
\end{document}